\documentclass{article}
\usepackage[utf8]{inputenc}
\usepackage{amsmath}
\usepackage{amssymb}
\usepackage{amsthm}

\newtheorem{theorem}{Theorem}
\newtheorem{corollary}{Corollary}

\title{Resolving the two envelope paradox}
\author{Nemo Semret\\
nemo@semret.org}
\date{January 10, 2021 \footnote{Revised January 28, 2021}}
\begin{document}
\maketitle
\begin{abstract}
Consider the following game: You are given two indistinguishable envelopes, each containing money. One contains twice as much as the other. You may pick one envelope and keep the money it contains. Having chosen an envelope, you are given the chance to switch envelopes. Should you switch? 

The intuitive answer is that it makes no difference, since you are equally likely to have picked the envelope with the higher or the lower amount. However, a naive expected value calculation implies you gain by switching, since you have $ 50\% $ chance of doubling and $ 50\% $ chance of halving your current winnings, and so if the first chosen envelope contains X, then switching gives an expected final value of $ (X/2 + 2X)/2 > X $.   That seems like  a paradox. 

We prove that the former is the correct answer, and show how the apparent "paradox" can be resolved. 
\end{abstract}

\section{Introduction}

The two envelope paradox is well known, and a number of approaches exist to explain it \cite{wiki}. Our approach is based on  probability theory. 

We formally define the problem, and prove that switching does not improve  expected value.  We then resolve the paradox by explaining where the "naive" calculation goes wrong. Finally we discuss how having some prior information changes the problem. 

\section{Expected value and optimal strategy}

Let Y be the base amount (i.e. the envelopes contain Y and 2Y),  it is a random variable with distribution \(F_1\), on a probability space \(\Omega\).  

Let X(Y) and X'(Y)  represent the amounts in the chosen envelope and the other envelope respectively.  To represent the choice of envelope, we have the distribution \(F_2\)  on a probability space with two events  \( \mathcal{E} =  \{\mathcal{E}_1, \mathcal{E}_2\} \).  X is a function \( X (Y(\omega), \epsilon ) \) where \( ( \omega, \epsilon) \in \Omega \times \mathcal{E} \).   When \( \epsilon =  \mathcal{E}_1  \), we happen to choose  the smaller envelope, so X = Y and X' = 2Y.  And when  \( \epsilon = \mathcal{E}_2 \),  X = 2Y and X' = Y. The choices are equally likely i.e. \( \mathbb{P}(\mathcal{E}_1) =  \int_{\mathcal{E}_1} dF_2 = 1/2\) and \(\mathbb{P}(\bar{\mathcal{E}_2}) = \int_{\bar{\mathcal{E}_2}} dF_2 = 1/2 \). 

Let V(Y) be the final value the player gets after deciding to switch or not. It has  distribution \(F_3\) on two events \( \{ \mathcal{S}, \bar{\mathcal{S}} \} \) representing respectively switching (V = X'), or not switching (V = X). For example if the strategy is to always switch, \( F_3 \) is deterministic with \( \mathbb{P}(\mathcal{S}) =  \int_{\mathcal{S}} dF_3 = 1, \mathbb{P}(\bar{\mathcal{S}}) = \int_{\bar{\mathcal{S}}} dF_3 = 0 \).  More generally, \(F_3\) is allowed to  depend on the observed value of X, as the player can decide to switch after seeing the value in the chosen envelope. But not on \(F_1 \) which the player has no information about.

Our key result is that the payoff is the same for all possible switching strategies. 

\begin{theorem}
\(E[V]  = \frac{3}{2} E[Y]\), for all possible switching strategies \(F_3 \). 
\end{theorem}

\begin{proof}
The expected value of V over all probability measures \( F_1 \), \( F_2  \) and  \( F_3 \) is: 
\begin{equation}
\label{ev}
E[V] = \int_{\Omega} E[V|Y] dF_1
\end{equation}
For a given Y, we have the choice of switching or not switching i.e.
\begin{align*}
E[V|Y]  = \int V(Y) dF_3  = \int_{\bar{\mathcal{S}}} X(Y)dF_3  + \int_{\mathcal{S}} X'(Y)  dF_3
\end{align*}
and expanding the terms to show the choice of envelope \(\mathcal{E}_1\) and \(\mathcal{E}_2\):
\begin{align*}
E[V|Y] & = \int_{\bar{\mathcal{S}}}  \left( \int_{\mathcal{E}_1} Y dF_2 + \int_{\mathcal{E}_2} 2Y dF_2 \right) dF_3 + \int_{\mathcal{S}}  \left( \int_{\mathcal{E}_1} 2Y dF_2 + \int_{\mathcal{E}_2} Y dF_2 \right) dF_3 \\
 &  = \int_{\mathcal{E}_1} \left( \int_{\bar{\mathcal{S}}} Y dF_3 + \int_{\mathcal{S}} 2Y dF_3 \right)  dF_2 + \int_{\mathcal{E}_2} \left( \int_{\bar{\mathcal{S}}} 2Y dF_3 + \int_{\mathcal{S}} Y dF_3 \right)  dF_2 
\end{align*}
Now,  since we don't know which envelope was chosen  (i.e. whether we are  in \(\mathcal{E}_1  \) or \( \mathcal{E}_2 \)) when we make the choice to switch or not (being  in \( \mathcal{S}\) or \( \bar{\mathcal{S}} \)), \( F_3 \)  is independent of \(F_2\), so we can take out the integrands:
\begin{equation}
\label{indep}
E[V|Y] =  \left( \int_{\bar{\mathcal{S}}} Y dF_3 + \int_{\mathcal{S}} 2Y dF_3 \right) \int_{\mathcal{E}_1}  dF_2 +  
\left( \int_{\bar{\mathcal{S}}} 2Y dF_3 + \int_{\mathcal{S}} Y dF_3 \right) \int_{\mathcal{E}_2} dF_2
\end{equation}
Thus, plugging in the actual values for \( F_2 \)
\begin{align*}
E[V|Y]  & =  \left( \int_{\bar{\mathcal{S}}} Y dF_3 + \int_{\mathcal{S}} 2Y dF_3 \right) \frac{1}{2}  +   \left( \int_{\bar{\mathcal{S}}} 2Y dF_3 + \int_{\mathcal{S}} Y dF_3 \right) \frac{1}{2} \\
 & = \frac{1}{2}  \left( \int_{\bar{\mathcal{S}}}\left( Y+2Y \right) dF_3  + \int_{\mathcal{S}} \left( 2Y+Y \right) dF_3 \right) \\
 & = \frac{3}{2}  Y \left( \int_{\bar{\mathcal{S}}} dF_3  + \int_{\mathcal{S}}   dF_3 \right)  = \frac{3}{2}  Y 
\end{align*}
Now plugging this back into \eqref{ev}, the overall expectation of V is 
\begin{equation}
\label{th1}
E[V]  = \frac{3}{2} \int_{\Omega}  Y   dF_1 =  \frac{3}{2} E[Y] 
\end{equation}
\end{proof}

Thus the expected value is always the same, regardless of the switching strategy \(F_3\), including never switching, and seeing X or not seeing X makes no difference.
Now, we can see there is no paradox. 

\section{Resolving the paradox}

\begin{corollary} The expected  value from switching to the second envelope is the same as the expected value of keeping the first, i.e. E[V] = E[X].
\end{corollary}
\begin{proof}
By definition of X,  
\begin{align*}
E[X|Y] =  \int_{\mathcal{E}_1} Y dF_2 + \int_{\mathcal{E}_2} 2Y dF_2 = \frac{1}{2} Y + \frac{1}{2} 2Y
\end{align*}
since \( \mathbb{P}(\mathcal{E}_1) = \mathbb{P}(\mathcal{E}_2) = 1/2 \).
Taking expectation over Y,  we get
\begin{equation}
\label{ex}
E[X] = \frac{3}{2} E[Y] 
\end{equation}
therefore from \eqref{th1},  \(E[X] =  E[V] \).
\end{proof}

The  "paradox" is that it may naively seem like \(E[V] >  E[X]\). To see how it arises and why it is incorrect, let's restate the E[V] calculation in terms of X: 
\begin{equation}
\label{evx}
E[V|X] = \int_{\bar{\mathcal{S}}} X dF_3 + \int_{\mathcal{S}}  \left( \int_{\mathcal{E}_1} 2X dF_2 + \int_{\mathcal{E}_2} \frac{1}{2}X dF_2 \right) dF_3 \\
\end{equation}
Naively treating X as a constant, since \( \mathbb{P}(\mathcal{E}_1) = \mathbb{P}(\mathcal{E}_2) = 1/2 \) it seems like   
\begin{equation}
\label{wrong}
E[V|X] \stackrel{?}{=}  \mathbb{P}(\bar{\mathcal{S}}) X + \mathbb{P}(\mathcal{S}) \frac{5}{4} X  
\end{equation}
which implies that  \(E[V|X] >  X\), i.e. any non-zero switching probability is a strict improvement. In fact always switching i.e. \( \mathbb{P}(\mathcal{S}) = 1 \) is optimal and gives a 25\% gain over never switching.  

The root of the apparent paradox is that  \eqref{wrong} is incorrect, because in \eqref{evx}, \(X \) is not actually a constant in  \( \mathcal{E} \). This is counter-intuitive because we can compute \( E[V|X] \)  after seeing the actual value of X, so X seems like it should be constant. But here we are evaluating on \( \mathcal{E} \),  the choice of envelope, which of course affects the value of X. More precisely \(X dF_2 \) is  \(X(Y, \epsilon) dF_2(\epsilon)\), i.e. it is not the same X in the two \(dF_2\)  integrals, since \( \epsilon \) is the variable being integrated on.  
The event \( \epsilon \) cannot by definition be treated as a constant when evaluating it's probabilities. 

But \( Y \) \emph{is} a constant in  \( \mathcal{E} \),  since the event \( \omega \in \Omega \) is determined.  So we can use the fact that  \( X(Y,  \mathcal{E}_1 ) = Y \), and \( X(Y,  \mathcal{E}_2)  = 2Y\). Also, as in \eqref{indep}, the choice \( F_3  \) is independent of \( F_2 \). Thus we get: 
\begin{align*}
E[V|X] & =  \mathbb{P}(\bar{\mathcal{S}}) X +  \mathbb{P}(\mathcal{S}) \left( \int_{\mathcal{E}_1}  2Y  dF_2 + \int_{\mathcal{E}_2}   
\frac{1}{2} 2Y dF_2 \right) \\
 & =  \mathbb{P}(\bar{\mathcal{S}}) X +  \mathbb{P}(\mathcal{S}) \left( \frac{1}{2}   2Y   + \frac{1}{2}\frac{1}{2} 2Y   \right) \\
& = \mathbb{P}(\bar{\mathcal{S}}) X +  \mathbb{P}(\mathcal{S}) \frac{3}{2} Y
\end{align*}
Now for the given X, we can take expectations over all values of Y, 
\begin{align*}
E[V|X] = E[E[V|X, Y]] & =  \mathbb{P}(\bar{\mathcal{S}}) X +  \mathbb{P}(\mathcal{S}) \frac{3}{2} E[Y]
\end{align*}
And, using \eqref{ex}, we see that: \( E[V] = \frac{3}{2} E[Y] \), i.e. there's no paradox.

A simple example to illustrate: Suppose we open the chosen envelope and see  X =\$100. Contrary to the naive estimate, we are not actually in a state where the other  envelope has an equal chance of containing \$50  or \$200.  Rather, we are in state where a hidden  Y has already been chosen and we are looking at Y or 2Y, with equal chance. To see this more clearly, imagine after the envelopes are filled, they are cloned into many instances of the game in parallel (not repeated!), and X is an average of the observed value. By ergodicity, the expected value in the one-shot game is the same as the average value in the  parallel games.   Since we expect to observe an average of X = 3Y/2, by seeing X = \$100, we "learn" that Y = 2X/3 = \$66.66... and the average value of switching or not switching remains \$100.

\section{Open vs closed envelope} 
Consider the variation of the problem based on  whether the player gets to see the value \(X \) or not before making the decision to switch. In all of the above, knowing the actual value of  of \(X \) does not change the optimal strategy. Thus in both the open and closed versions, the answer remains that switching makes no difference.  Indeed we don't know anything about \(F_1 \) so knowing one value doesn't help us decide if \( X \) is \(Y \) or \(2Y\). 

\section{Prior information on distribution} 
If we (the player) have some prior knowledge of \(F_1\), then, given X, we may know if it's more likely to be Y or 2Y, i.e. if we are in \(\mathcal{E}_1  \) or \( \mathcal{E}_2 \), which we can use in deciding to switch or not, i.e. \(F_3\) can be a function of \(\epsilon \). Therefore the step \eqref{indep} where we factor out \(dF_2 \) from \(dF_3 \) is no longer valid and it is no longer true that all strategies have the same \(E[V] \). In fact switching does sometimes lead to gains, and knowing the value of \(X \) before deciding makes a difference too.

For example, if we know  the average E[Y], then we would use the strategy: switch if and only if \( X < \frac{3}{2} E[Y]  \).   Another example is, if we know \( Y_{max} \) the largest possible value of Y, then when \( X > Y_{max}  \),  naturally we should never switch because we know for sure that we are in \(\mathcal{E}_2  \).  Similarly, if  we know \( Y_{min} \) the smallest  possible value of Y, then  when \( X < 2Y_{min}  \),  we should always  switch  because we know for sure that we are in \(\mathcal{E}_1  \). 

More broadly, if we know \( F_1 \), after seeing \( X \), we can estimate \( \mathbb{P}(\mathcal{E}_1 | X) \) and \( \mathbb{P}(\mathcal{E}_2 | X) \),  and choose a mixed strategy \(F_3\) whereby the probability of not switching, \( \mathbb{P}(\bar{\mathcal{S}}) \), is higher when \( \mathbb{P}(\mathcal{E}_2|X) \) is higher, and  \( \mathbb{P}(\mathcal{S}) \) is higher if \( \mathbb{P}(\mathcal{E}_1 |X) \) is higher. 

In general, the switching strategy can be optimized to take advantage of any prior information about \(F_1\). A few interesting cases are covered in \cite{mcdon}.
\section*{Acknowledgement}
I would like to thank Jacob Eliosoff for pointing me to this problem and for helpful discussions. 

\end{document}